\documentclass{amsart} 
\usepackage{amscd,amssymb,latexsym,subfigure,verbatim,hyperref,epsfig,
amsmath, supertabular,xypic}
\usepackage{calc}
\usepackage{amsmath}
\usepackage{graphicx}

\begin{document}

\newcommand{\mmbox}[1]{\mbox{${#1}$}}
\newcommand{\proj}[1]{\mmbox{{\mathbb P}^{#1}}}
\newcommand{\affine}[1]{\mmbox{{\mathbb A}^{#1}}}
\newcommand{\Ann}[1]{\mmbox{{\rm Ann}({#1})}}
\newcommand{\caps}[3]{\mmbox{{#1}_{#2} \cap \ldots \cap {#1}_{#3}}}
\newcommand{\N}{{\mathbb N}}
\newcommand{\Z}{{\mathbb Z}}
\newcommand{\Q}{{\mathbb Q}}
\newcommand{\R}{{\mathbb R}}
\newcommand{\K}{{\mathbb K}}
\newcommand{\p}{{\mathbb P}}
\newcommand{\A}{{\mathcal A}}
\newcommand{\CC}{{\mathcal C}}
\newcommand{\C}{{\mathbb C}}
\newcommand{\OO}{{\mathcal O}}
\newcommand{\TT}{{\mathcal T}}

\newcommand{\Tor}{\mathop{\rm Tor}\nolimits}
\newcommand{\Ext}{\mathop{\rm Ext}\nolimits}
\newcommand{\Hom}{\mathop{\rm Hom}\nolimits}
\newcommand{\im}{\mathop{\rm Im}\nolimits}
\newcommand{\reg}{\mathop{\rm reg}\nolimits}
\newcommand{\syz}{\mathop{\rm syz}\nolimits}
\newcommand{\sing}{\mathop{\rm sing}\nolimits}
\newcommand{\pdim}{\mathop{\rm pdim}\nolimits}
\newcommand{\Sing}{\mathop{\rm Sing}\nolimits}
\newcommand{\HP}{\mathop{\rm HP}\nolimits}
\newcommand{\HS}{\mathop{\rm HS}\nolimits}
\newcommand{\Der}{\mathop{\rm Der}\nolimits}
\newcommand{\rank}{\mathop{\rm rank}\nolimits}
\newcommand{\supp}{\mathop{\rm supp}\nolimits}
\newcommand{\coker}{\mathop{\rm coker}\nolimits}
\sloppy
\newtheorem{defn0}{Definition}[section]
\newtheorem{prop0}[defn0]{Proposition}
\newtheorem{conj0}[defn0]{Conjecture}
\newtheorem{thm0}[defn0]{Theorem}
\newtheorem{lem0}[defn0]{Lemma}
\newtheorem{corollary0}[defn0]{Corollary}
\newtheorem{example0}[defn0]{Example}

\newenvironment{defn}{\begin{defn0}}{\end{defn0}}
\newenvironment{prop}{\begin{prop0}}{\end{prop0}}
\newenvironment{conj}{\begin{conj0}}{\end{conj0}}
\newenvironment{thm}{\begin{thm0}}{\end{thm0}}
\newenvironment{lem}{\begin{lem0}}{\end{lem0}}
\newenvironment{cor}{\begin{corollary0}}{\end{corollary0}}
\newenvironment{exm}{\begin{example0}\rm}{\end{example0}}

\newcommand{\defref}[1]{Definition~\ref{#1}}
\newcommand{\propref}[1]{Proposition~\ref{#1}}
\newcommand{\thmref}[1]{Theorem~\ref{#1}}
\newcommand{\lemref}[1]{Lemma~\ref{#1}}
\newcommand{\corref}[1]{Corollary~\ref{#1}}
\newcommand{\exref}[1]{Example~\ref{#1}}
\newcommand{\secref}[1]{Section~\ref{#1}}

\newcommand{\poina}{\pi({\mathcal A}, t)}
\newcommand{\poinc}{\pi({\mathcal C}, t)}

\newcommand{\std}{Gr\"{o}bner}
\newcommand{\jq}{J_{Q}}



\title[Logarithmic vector fields and curve configurations]{Logarithmic vector fields for quasihomogeneous curve configurations in $\p^2$}

\author{Hal Schenck}
\thanks{Schenck supported by NSF 1068754, NSA H98230-11-1-0170}\address{Schenck: Mathematics Department \\ University of Illinois \\
   Urbana \\ IL 61801\\USA}
\email{schenck@math.uiuc.edu}

\author{Hiroaki Terao}
\address{Terao: Department of Mathematics \\ Hokkaido University \\
  Sapporo \\ 060-0310 \\Japan}
\email{hterao00@za3.so-net.ne.jp}

\author{Masahiko Yoshinaga}
\address{Yoshinaga: Department of Mathematics \\  Hokkaido University \\
  Sapporo \\ 060-0310 \\Japan}
\email{yoshinaga@math.sci.hokudai.ac.jp}


\begin{abstract}
\noindent Let ${\mathcal A}= \bigcup_{i=1}^r C_i \subseteq \mathbb{P}^2_{\mathbb{C}}$ be a collection of smooth plane curves, such that each 
singular point is quasihomogeneous. We prove that if 
$C$ is a smooth curve such that each singular point of 
${\mathcal A} \cup C$ is also quasihomogeneous, 
then there is an elementary modification of rank two bundles,
which relates the $\OO_{\p^2}$--module $\Der(\log {\mathcal A})$ 
of vector fields on $\p^2$ tangent to ${\mathcal A}$ to 
the module $\Der(\log {\mathcal A}\cup C)$. 
This yields an inductive tool for studying the splitting 
of the bundles $\Der(\log {\mathcal A})$ and $\Der(\log {\mathcal A}\cup C)$, 
depending on the geometry of the divisor ${\mathcal A}|_{C}$ on $C$.
\end{abstract}
\maketitle

\section{Introduction}\label{sec:intro}
For a divisor $Y$ in a complex manifold $X$, Saito \cite{slog}
introduced the sheaves of logarithmic vector fields and logarithmic 
one forms with pole along $Y$: 
\begin{defn}The module of logarithmic vector fields is the sheaf of
$\mathcal{O}_X$--modules
\[
\Der(\log Y)_p = \{ \theta \in \Der_{\C}(X)| \theta(f) \in \langle f \rangle \},
\]
where $f \in \mathcal{O}_{X,p}$ is a local defining equation for $Y$ at $p$. 
\end{defn}
Saito's work generalized earlier work of Deligne \cite{d}, where the
situation was studied for $Y$ a normal crossing divisor.
If $\{x_1,\ldots,x_d\}$ are local coordinates at a 
point $p \in X$ and $Y$ has local equation $f$, 
then $\Der(\log Y)_p$ is the kernel of the evaluation map 
$\theta \mapsto \theta(f) \in \mathcal{O}_{Y,p}$, so there is a short
exact sequence
\[
0 \longrightarrow \Der(\mbox{log }Y) \longrightarrow \mathcal{T}_X  
\longrightarrow J_Y(Y) \longrightarrow 0,
\]
where $J_Y$ is the Jacobian scheme, defined locally at $p$ by 
$\{\partial f/\partial_{x_1}, \ldots, \partial f/\partial_{x_d}\}$.
Saito shows that $\Der(\log Y)_p$ is a free $\OO_{X,p}$ module iff 
there exist $d$ elements 
\[
\theta_i = \sum\limits_{j=1}^d f_{ij}\frac{\partial}{\partial x_j} \in \Der(\log  Y)_p
\]
such that the determinant of the matrix $[f_{ij}]$ is a nonzero
constant multiple of the local defining equation for $Y$; this
is basically a consequence of the Hilbert-Burch theorem. 
A much studied version of this construction occurs when 
$Y = V(F) \subseteq \p^d$ is a reduced hypersurface; $V(F)$ may
also be studied as a hypersurface in $\mathbb{C}^{d+1}$. 
In particular, if $X=\mathbb{C}^{d+1}$ and $F \in S =\C[x_0,\ldots, x_d]$ 
is homogeneous, we write
\begin{defn}$D(V(F)) = \{ \theta \in Der_{\C}(S)| \theta(F) \in \langle F \rangle \}.$
\end{defn}
Since $F$ is homogeneous, $D(V(F))$ is a graded $S$--module, hence 
gives rise to an associated sheaf on $\p^d$. The kernel of the 
evaluation map $\theta \mapsto \theta(F)$
\[
D(V(F)) \rightarrow S
\]
consists of the syzygies on the Jacobian ideal of $F$, which we denote
$D_0(V(F))$. Since the Euler vector 
field $\sum x_i \partial/\partial x_i$ gives a surjection 
\[
D(V(F)) \rightarrow \langle F \rangle \rightarrow 0,
\]
we can split the map, hence
\[
D(V(F))\simeq D_0(V(F)) \oplus S(-1).
\]
In particular, if $F$ has degree $n$, then there is an 
exact sequence
\[
0 \longrightarrow D_0(V(F)) \longrightarrow S^{d+1} 
\stackrel{\theta \mapsto \theta(F)}{\longrightarrow} S(n-1) \longrightarrow S(n-1)/J_F
 \longrightarrow 0.
\]
This shows that $D(V(F))$ and the associated sheaf are second syzygies, 
so when $d=2$, $D(V(F))$ is a vector bundle on $\p^2$. Since the depth of
$D(V(F))$ is at least two, $D(V(F))$ is $\Gamma_*$ of the associated sheaf;
we use $D(V(F))$ to denote both the $S$--module and associated sheaf.
Tensoring $D(V(F))$ with $\OO_{\p^d}(1)$ yields the commutative diagram
\xymatrixrowsep{30pt}
\xymatrixcolsep{30pt}
\[
\xymatrix{
 & 0 \ar[d] & 0 \ar[d] \\
 0 \ar[r] & \OO_{\p^d}  \ar[r] \ar[d]^{\begin{tiny}\left[ \!
\begin{array}{c}
x_0\\
\vdots\\
x_d
\end{array}\! \right]\end{tiny}}   & \OO_{\p^d}  \ar[r] \ar[d]^{\begin{tiny}\left[ \!
\begin{array}{c}
x_0\\
\vdots\\
x_d
\end{array}\! \right]\end{tiny}} & 0 \ar[d]\\
0 \ar[r] & D(V(F))(1) \ar[r] \ar[d] & \OO_{\p^d}^{d+1}(1)  \ar[r]^{\gamma} \ar[d] & \OO_{V(F)} \ar[d] \\
0 \ar[r] & D(V(F))/E(1) \ar[r] \ar[d]  & \TT_{\p^d}  \ar[r] \ar[d]  & \OO_{V(F)}  \ar[d] \\
 & 0 & 0 & 0.
}
\]
The map $\gamma$ sends $\theta = \sum f_i \partial/\partial x_i$ to $\theta(F)$, so $\gamma$ surjects onto $J_F(n)$. Hence, 
\begin{equation}\label{DD}
\Der(\mbox{log }V(F)) \simeq D(V(F))/E(1) \simeq D_0(V(F))(1).
\end{equation}
In the case of generic hyperplane arrangements this is noted in 
\cite{ms}, we include it here to make precise
the relationship. A major impetus in studying $D(V(F))$ comes from 
the setting of hyperplane arrangements. If $F$ is a product of 
distinct linear forms, then write $V(F) = {\mathcal A}$. 
Terao's theorem \cite{t} shows that in this setting, if 
$D({\mathcal A})$ is a free $S$--module, with $D({\mathcal A}) \simeq
\oplus S(-a_i)$, then 
\[
\sum h^i(\C^{d+1}\setminus \mathcal{A},\Q)t^i = \prod(1+a_it).
\]
\subsection{Addition-Deletion theorems}
A central tool in the study of hyperplane arrangements
is an inductive method due to Terao. For a hyperplane 
arrangement ${\mathcal A}$ and choice of $H \in {\mathcal A}$, set
$${\mathcal A'} = {\mathcal A} \setminus H \mbox{ and }{\mathcal A''} =
{\mathcal A}|_H.$$ The collection
$({\mathcal A}', {\mathcal A}, {\mathcal A}'')$ is called a {\em triple},
and yields a left exact sequence
$$ 0 \longrightarrow D({\mathcal A}')(-1) \stackrel {\cdot H}{\longrightarrow}
D({\mathcal A}) \longrightarrow D({\mathcal A}'').$$ 
Freeness of a triple is related via:
\begin{thm}\label{thm:teraoAD}$[$Terao, \cite{t2}$]$
Let  $({\mathcal A}', {\mathcal A}, {\mathcal A}'')$ be a triple. 
Then any two of the following imply the third
\begin{itemize}
\item $D({\mathcal A})\simeq \oplus_{i=1}^n S(-b_i)$
\item $D({\mathcal A}')\simeq S(-b_n +1)\oplus_{i=1}^{n-1} S(-b_i)$
\item $D({\mathcal A}'')\simeq \oplus_{i=1}^{n-1} S(-b_i)$
\end{itemize}
\end{thm}
For a triple with ${\mathcal A} \subseteq \mathbb{P}^2$, \cite{cmh} shows
that after pruning the Euler derivations and sheafifying, there is an exact sequence
\[
0 \longrightarrow {\mathcal D_0}'(-1) \longrightarrow {\mathcal D_0}
  \longrightarrow i_*{\mathcal D_0}'' \longrightarrow 0,
\]
where $i: H \hookrightarrow \mathbb{P}^2$; $i_*{\mathcal D_0}''\simeq {\mathcal O}_H(1-|{\mathcal A}''|).$
In \cite{cmh2}, a version of this exact sequence (and associated addition-deletion theorem)
is shown to hold for arrangements of lines and conics in $\p^2$ having all singularities 
quasihomogeneous. In related work, Dimca-Sticlaru study the
Milnor algebra of nodal curves in \cite{ds}. 
\subsection{Statement of results}
This paper is motivated by the following example:
\begin{exm}\label{firstex}
The braid arrangement $\operatorname{A}_3$ is depicted below;
\begin{figure}[ht]
\begin{center}
\epsfig{file=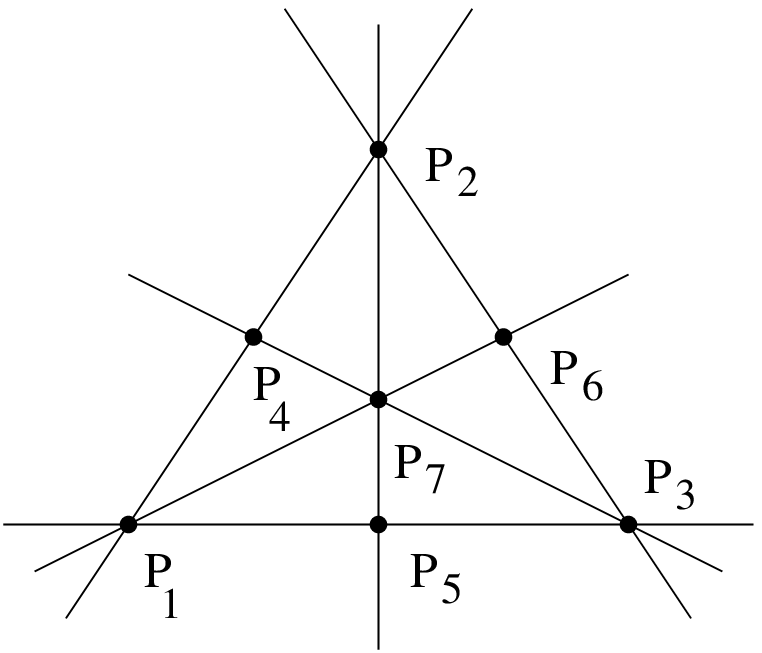,height=1.6in,width=1.6in}
\end{center}
\caption{\textsf{The $\operatorname{A}_3$-arrangement}}
\end{figure}
$D(\operatorname{A}_3)$ is free and is \newline isomorphic to $S(-1) \oplus S(-2)\oplus S(-3)$. 
There is a three dimensional family of cubics passing through the seven singular points
of $\operatorname{A}_3$. A generic member $C$ of this family is smooth, and 
$D(\operatorname{A}_3 \cup C) \simeq S(-1) \oplus S(-4)\oplus S(-4).$ 
Our main result is Theorem~\ref{main}, which combined with Proposition~\ref{cmain} explains this example.
\end{exm}
\begin{thm}\label{main}
Let ${\mathcal A}= \bigcup_{i=1}^r C_i$ and 
 ${\mathcal A} \cup C$ be collections of 
smooth curves in $\p^2$, such that all singular points
of ${\mathcal A}$ and ${\mathcal A} \cup C$ are quasihomogeneous. Then 
\[
0 \longrightarrow \Der(\log {\mathcal A})(-n) 
\longrightarrow \Der(\log {\mathcal A} \cup C) \longrightarrow \mathcal{O}_{C}(-K_C -R) \longrightarrow 0
\]
is exact, where $R = ({\mathcal A} \cap C)_{red}$ is the reduced 
scheme of $C \cap {\mathcal A}$ and $\deg(C) = n$. 
\end{thm}
\noindent We prove the theorem in \S 3. For addition-deletion arguments, we will need
\begin{prop}\label{cmain}
Suppose $0 \longrightarrow A \longrightarrow B \longrightarrow C \longrightarrow 0$ 
is an exact sequence of graded $S=k[x_0,\ldots,x_n]$--modules, with $A$ and $B$ of rank two 
and projective dimension at most one. Then any two of the following imply the third
\begin{enumerate}
\item $A$ is free with generators in degrees $\{a,b\}$.
\item $B$ is free with generators in degrees $\{c,d\}$.
\item $C$ has Hilbert series $\frac{t^c+t^d-t^{a}-t^{b}}{(1-t)^{n+1}}$.
\end{enumerate}
\end{prop}
\begin{proof}
That (1) and (2) imply (3) is trivial. If (1) and (3) hold and $B$ is 
not free, then $\pdim(B)=1$, so $B$ has a minimal
free resolution of the form
\[
0 \longrightarrow F_1 \longrightarrow F_0 \longrightarrow B  \longrightarrow 0.
\]
Note that 
$F_1 = \oplus S(-e_i)$ and $F_0 =  \oplus S(-e_i) \oplus S(-c) \oplus S(-d)$,
since by additivity the Hilbert series of $B$ is 
$\frac{t^c+t^d}{(1-t)^{n+1}}$. Without loss of generality suppose $d \ge c$. 
If the largest $e_i > d$, then since the resolution is minimal, no term of
$F_1$ can map to the generator in degree $e_i$, which forces $B$ to have a
free summand of rank one. Since $B$ has rank two this forces $B$ to be free.
Next suppose the largest $e_i = d$. Since the Hilbert series of $B$ is 
$\frac{t^c+t^d}{(1-t)^{n+1}}$, $B$ has a minimal generator of degree $d$.
But then no element of $F_1$ can be a relation involving that generator,
and again $B$ has a free summand. This obviously also works when the 
largest $e_i$ is less than $d$, and shows that  (1) and (3) imply (2). 
The argument that (2) and (3) imply (1) is similar.
\end{proof}
If a bundle of the form $\Der(\log {\mathcal A})$ splits as a sum of line bundles (is free), then 
applying $\Gamma_*$ to the short exact sequence of Theorem~\ref{main} 
yields a short exact sequence of modules, since 
$H^1(\oplus \OO_{\p^2}(a_i)) = 0$. Then by Proposition~\ref{cmain}, on
$\p^2$ the freeness of $\Der(\log {\mathcal A} \cup C)$ follows if 
appropriate numerical conditions hold. In contrast to arrangements of 
rational curves (where the Hilbert series of $\Gamma_*(\OO_C(-K_C-R))$ 
depends only on the degree of $R$, since $C \simeq \mathbb{P}^1$), 
for curves of positive genus 
the Hilbert series of $\Gamma_*(\OO_C(-K_C-R))$ depends on subtle geometry.
\section{Quasihomogeneous plane curves}
Let $C=V(Q)$ be a reduced (but not necessarily irreducible) curve in $\mathbb{C}^2$, let
$(0,0) \in C$, and let $\mathbb{C}\{x,y\}$ denote the ring of 
convergent power series.
\begin{defn}
The Milnor number of $C$ at $(0,0)$ is 
\[
\mu_{(0,0)}(C) = \dim_\mathbb{C}\mathbb{C}\{x,y\}/\langle\frac{\partial f}{\partial
x}\mbox{, }\frac{\partial f}{\partial y}\rangle.
\]
\end{defn}
\noindent To define $\mu_p$ for an arbitrary point $p$, we translate so
that $p$ is the origin. 
\begin{defn}
The Tjurina number of $C$ at $(0,0)$ is 
\[
\tau_{(0,0)}(C) = \dim_{\C} \C \{x,y\}/\langle\frac{\partial f}{\partial
x}\mbox{, }\frac{\partial f}{\partial y}\mbox{, }f\rangle.
\]
\end{defn}
\begin{defn}
A singularity is quasihomogeneous iff there exists a holomorphic 
change of variables so the defining equation becomes weighted
homogeneous; $f(x,y) = \sum c_{ij}x^{i} y^{j}$ is
weighted homogeneous if there exist rational numbers $\alpha, \beta$
such that $\sum c_{ij}x^{i \cdot \alpha} y^{j \cdot \beta}$ is homogeneous.
\end{defn}
In \cite{s}, Saito shows that if $f$ is a convergent power series
with isolated singularity at the origin, then $f$ is in the ideal
generated by the partial derivatives if and only if $f$ is
quasihomogeneous. As noted in \S 1.3 of \cite{cmh2}, 
if all the singular points of $V(Q) \subseteq \p^2$ are quasihomogeneous, then
\[
\text{deg}(J_Q) = \!\!\!\!\sum\limits_{p \in \text{Sing}(V(Q))}\!\!\!\!\mu_p(Q).
\]
\begin{lem}\label{lem:ADmilnors}$[$\cite{w}, Theorem 6.5.1$]$
Let $X$ and $Y$ be two reduced plane curves with no common component, 
meeting at a point $p$. 
Then 
\[
\mu_p(X \cup Y) = \mu_p(X) + \mu_p(Y) +2(X \cdot Y)_p -1,
\]
where $(X \cdot Y)_p$ is the intersection number of $X$ and $Y$ at $p$.
\end{lem}

\begin{prop}\label{HP}
 Let $\mathcal A$ and $\mathcal A \cup C$ be quasihomogeneous,
with $C=V(f)$ of degree $n$ and $\mathcal A = V(Q)$ with $Q$ of degree $m$. 
Then the Hilbert polynomial $\HP(coker(f),t)$ of the cokernel of the 
multiplication map
\[
 0\longrightarrow D(\mathcal A)(-n)/E \stackrel{\cdot f}{\longrightarrow} D(\mathcal A \cup C)/E
\]
is $nt+3(1-g_C)-n-k$, where $k=|C \cap \mathcal A|$.
\end{prop}

\begin{proof} By Equation~\ref{DD} and the exact sequences
$$0\longrightarrow D_0(\mathcal A \cup C)\longrightarrow S^3\longrightarrow S(m+n-1)\longrightarrow S(m+n-1)/J_{A \cup C}
\longrightarrow 0,$$ 
$$0\longrightarrow D_0(\mathcal A)(-n)\longrightarrow S^3(-n)\longrightarrow S(m-n-1)\longrightarrow S(m-n-1)/J_{A }
\longrightarrow 0,$$ 
it follows that
\[
\HP(D_0(\mathcal A \cup C),t) = 3{{t+2}\choose{2}}-{{t+1+m+n}\choose{2}}+\deg(J_{A \cup C})
\]
\[
\HP(D_0(\mathcal A)(-n),t) = 3{{t+2-n}\choose{2}}-{{t+1+m-n}\choose{2}}+\deg(J_{A }),
\]
so that $\HP(D_0(\mathcal A \cup C),t)-\HP(D_0(\mathcal A)(-n),t)$ is equal to
\[
\deg(J_{A \cup C})-\deg(J_{A })+nt-n(2m+1)+\frac{3}{2}(3-n)(n).
\]
Since $C$ is smooth of degree $n$, $\frac{3}{2}(3-n)(n) = 3(1-g_C)$.
To compute $\deg(J_{A \cup C})-\deg(J_{A})$, note that since all singularities 
of $\mathcal A$ and $\mathcal A \cup C$ are quasihomogeneous,
\[
\deg(J_{A \cup C}) =\!\!\!\! \sum\limits_{p \in \Sing(A \cup C)}\!\!\!\! \mu_p(\mathcal A \cup C) \mbox{ and }
\deg(J_{A}) =\!\!\!\! \sum\limits_{p \in \Sing(\mathcal A)}\!\!\!\! \mu_{p}(\mathcal A). 
\]
Let $\alpha$ be the sum of Milnor numbers of points off $C$, 
so 
\[
\deg(J_{A \cup C})=\alpha+\sum_{p \in C \cap \mathcal A}\mu_p(\mathcal A \cup C).
\]
Since $\mu_p(C) =0$, by Lemma \ref{lem:ADmilnors}, the previous quantity equals
\[
\alpha+\sum_{p \in C \cap \mathcal A}(\mu_p(\mathcal A)+ 2(C \cdot \mathcal A)_p -1).
\]
As $\deg(J_{A})=\alpha+\sum_{p \in C \cap \mathcal A} \mu_p(\mathcal A)$ and $\mid C \cap
\mathcal A\mid = k$, we obtain:
\[
\deg(J_{A \cup C})-\deg(J_{A})=2\!\!\sum_{p \in C \cap \mathcal A}\!\!(C \cdot \mathcal A)_p -k.
\]
By  Bezout's theorem, 
\[
\sum_{p \in C \cap \mathcal A}(C \cdot \mathcal A)=mn, \mbox{ so }\deg(J_{A \cup C})-\deg(J_{A})=2mn-k,
\]
hence the Hilbert polynomial of the cokernel is 
\[
nt-n(2m+1)+3(1-g_C) +2mn -k  =  nt-n-k+3(1-g_C).
\]
\end{proof}

\section{Main Theorem}
We now prove Theorem~\ref{main}. First we show that 
the sheaf associated to the cokernel of the multiplication map
\[
 0\longrightarrow \Der(\log A)(-n) \stackrel{\cdot f}{\longrightarrow} \Der(\log A \cup C)
\]
is isomorphic to $\mathcal{O}_{C}(D)$, where $D$ is a divisor of degree $3n-n^2-k$. Consider the commuting diagram below
\[
\xymatrix{ 
          & 0 \ar[d]  & 0 \ar[d]  & 0 \ar[d] \\
 0 \ar[r] & \Der(\log A)(-n)  \ar[r] \ar[d]^{\cdot f}   & \TT_{\p^d}(-n)  \ar[r] \ar[d]^{\cdot f} & J_A(-n) \ar[d]^{\cdot f} \ar[r]& 0\\
 0 \ar[r] & \Der(\log A \cup C)  \ar[r] \ar[d]   & \TT_{\p^d}  \ar[r] \ar[d] & 
J_{A\cup C} \ar[d]\ar[r]& 0\\
0 \ar[r] & \coker(f) \ar[r]\ar[d]  & \TT_{\p^d}/f\cdot \TT_{\p^d} \ar[r]\ar[d] & J_{A\cup C}/J_A(-n) \ar[r]\ar[d] &0\\
  & 0   & 0   & 0
}
\]

Exactness of the top two rows follows from the definition, and since all 
the modules are torsion free, multiplication by $f$ gives an inclusion. 
Exactness of the bottom row then follows from the snake lemma.
The exact sequence
\[
0 \longrightarrow S \longrightarrow S^3(1)  \longrightarrow \Gamma_*(\TT_{\p^2}) \longrightarrow 0,
\]
shows that the Hilbert polynomial of $\Gamma_*(\TT_{\p^2} / f\cdot \TT_{\p^2})$ is $2nt +6n-n^2$. Since 
\[
\TT_{\p^2} / f\cdot \TT_{\p^2} \simeq \TT_{\p^2} \otimes \OO_C,
\]
and $\TT_{\p^2} \otimes \OO_C$ is a locally free rank two $\OO_C$--module, 
$\coker(f)$ is a torsion free submodule of $\TT_{\p^2} \otimes \OO_C$, hence
locally free on $C$, of rank one or two. Since 
\[
\Der(\log \mathcal{A} \cup C) \simeq D_0(\mathcal{A} \cup C)(1),
\]
by Proposition~\ref{HP}, the Hilbert polynomial of $\Gamma_*(\coker(f))$ 
is $nt-k+3(1-g_C)$, so $\coker(f) \simeq \OO_C(D)$. To determine the
degree of $D$, we compute

$\begin{array}{ccc}
\HP(\Gamma_*(\coker(f)), t) & = & h^0( \OO_C(D+tH)), \mbox{ } t\gg 0\\
                            & = & \deg(D+tH)+1-g_C\\
                            & = & \deg(D) + nt + 1 -g_C.
\end{array}$

\noindent Equating this with the previous expression shows 
that $\deg(D) = 2-2g_C -k$. By adjunction,
\[
2g_C-2 = C(C+K_{\p^2}) = nH(nH-3H) = n^2-3n,
\]
so $\deg(D) = 3n-n^2-k$. Notice this shows the left hand column is 
an elementary modification of bundles. To conclude, consider the 
short exact sequence
\[
0 \longrightarrow \TT_C \longrightarrow \TT_{\p^2}\otimes \OO_C \longrightarrow N_{C/\p^2} \longrightarrow 0.
\]
Since $\OO_C(D)$ comes from the restriction of $\Der(\log \mathcal{A} \cup C)$
to $C$, it must actually be a subbundle of $\TT_C$, which by 
adjunction is isomorphic to $\OO_C((3-n)H)$. For the same reason, 
sections must vanish at points of $\mathcal{A} \cap C$, so that
$\OO_C(D) \subseteq  \OO_C((3-n)H -R)$. But the degree of this
last bundle is $3n-n^2-k$, so we have equality, which concludes 
the proof. $\Box$ \newline

\subsection{Castelnuovo-Mumford regularity}
Theorem~\ref{main} yields bounds on the Castelnuovo-Mumford 
regularity of logarithmic vector bundles and, by dualizing, for
logarithmic one forms for quasihomogeneous curve arrangements.
\begin{defn} A coherent sheaf $\mathcal F$ on $\mathbb P^d$ is $j-$regular iff
$H^i\mathcal F(j-i)=0$ for every $i\geq 1$. The smallest number $j$
  such that $\mathcal F$ is $j$-regular is $\reg(\mathcal F)$.
\end{defn}

\begin{lem}\label{lem:regbounds} 
With the hypotheses of Theorem~\ref{main}, 
\[
\reg(\mathcal D_0({\mathcal A \cup C}))\leq \max\{\reg(\mathcal D_0({\mathcal A}))+n, 2n-4+\frac{k}{n}\}.
\]
\end{lem}
\begin{proof} The short exact sequence
\[
0 \longrightarrow D_0({\mathcal A})(-n)\longrightarrow D_0({\mathcal A}\cup C)
\longrightarrow \mathcal{O}_{C}(-K_C -H -R)\longrightarrow 0
\]
gives a long exact sequence in cohomology, so if $ D_0({\mathcal A})$ 
is $a$-regular, then 
\[
h^1( D_0({\mathcal A})(a-1))=0 = h^2( D_0({\mathcal A})(a-2)).
\]
So if $t-n-1 \ge a-1$ and $t-n-2 \ge a-2$ we have that 
\[
h^1(D_0({\mathcal A})(t-n-1))=0 = h^2(D_0({\mathcal A})(t-n-2)).
\]
This gives vanishings if $t-n \ge a$, that is, if 
$t \ge \reg D_0({\mathcal A}) +n$. The result will follow if
\[
h^1\mathcal{O}_{C}((t-2)H-K_C -R)=h^0\mathcal{O}_{C}((2-t)H+2K_C+R)=0.
\]
Using that $K_C = (n-3)H$, this holds if $\deg ((t-2)H +2K_C+R) <0$, that 
is, when
\[
t > 2n-4+\frac{k}{n}
\]
The result follows.
\end{proof}
The previous proof shows that the Hilbert function of 
$\Gamma_{*}\mathcal{O}_{C}(-K_C -H -R)$ is equal to the Hilbert 
polynomial when $t> 2n-4+\frac{k}{n}$.
\begin{prop}\label{HS}
With the hypotheses of Theorem~\ref{main}, the Hilbert function of 
$\Gamma_{*}\mathcal{O}_{C}(-K_C -H -R)$ is $nt+3(1-g_C)-n-k$
for 
\[
t < n-2 + \frac{k}{n} \mbox{ or }t > 2n-4+\frac{k}{n}.
\]
\end{prop}
\begin{proof} 
If $t < n-2 + \frac{k}{n}$, then the degree of $(t-1)H-K_C-R$
is negative, so there can be no sections. 
\end{proof}
\section{Examples}
\begin{exm}\label{firstex2}
We analyze Example~\ref{firstex} in more detail. Since $C$ is a cubic curve,
$g_C=1$ and $K_C \simeq \OO_C$. Since $C$ meets every line of $\mathcal{A}$ 
in three points, $k = |C \cap \mathcal{A}|= 7$. By Proposition~\ref{HS}, the
Hilbert polynomial and Hilbert function of $\Gamma_{*}\mathcal{O}_{C}(-K_C-H -R)$
agree for $ t < \frac{10}{3}$ and $t > \frac{13}{3}$, and so applying 
Proposition~\ref{HP} and Theorem~\ref{main} we have
\vskip .1in
\begin{center}
\begin{supertabular}{|c|c|c|c|c|c|c|c|c|}
\hline $t$ & $0$ & $1$ & $2$ & $3$ & $4$ & $5$ & $6$& $7$ \\
\hline $h^0((t-1)H-K_C-R)$ & $0$ & $0$  & $0$ & $0$ & ?  & $5$ & $8$  & $11$\\
\hline
\end{supertabular}
\end{center}
\vskip .1in
Now, $H^0((4-1)H-R)$ consists of cubics through the seven singular points of  $\mathcal{A}$, 
and this space has dimension $2$, since $C$ is itself one of the three cubics, so is not counted. 
Thus, in this example the Hilbert polynomial $3t-10$ agrees with the Hilbert function for all $t$, and
\[
\HS(\Gamma_{*}\mathcal{O}_{C}(-K_C -H -R),t) = \frac{2t^4+t^5}{(1-t)^2}= \frac{2t^4-t^5-t^6}{(1-t)^3}.
\]
Terao's result \cite{t3} on reflection arrangements shows that 
$D(\operatorname{A}_3) \simeq S(-1)\oplus S(-2) \oplus S(-3)$, so
\[
\Der(\mbox{log}\operatorname{A}_3(-3))\simeq S(-5)\oplus S(-6).
\]
Taking global sections in Theorem~\ref{main} and applying 
Proposition~\ref{cmain}, we find that $\Der(\mbox{log }\operatorname{A}_3 \cup C)$ is free, with
\[
\begin{array}{ccc}
\HS(\Der(\mbox{log}\operatorname{A}_3 \cup C))
& =&\frac{t^5+t^6}{(1-t)^3} + \frac{2t^4-t^5-t^6}{(1-t)^3}\\
 & = &\frac{2t^4}{(1-t)^3}. 
\end{array}
\]
\end{exm}
\begin{exm}\label{secondex}
The reflection arrangement $\operatorname{B}_3$ consists of the nine planes of symmetry
of a cube in $\R^3$. The intersection of $\operatorname{B}_3$ with the affine chart 
$U_z$ is depicted below (this does not show the line at infinity $z=0$).
\begin{figure}[ht]
\subfigure{%
\label{fig:B3-a}%
\begin{minipage}[t]{0.3\textwidth}
\setlength{\unitlength}{0.78cm}
\begin{picture}(5,4.8)(-0.2,-1)
\multiput(1,0)(1,0){2}{\line(1,1){3}}
\multiput(4,0)(1,0){2}{\line(-1,1){3}}
\multiput(2.5,0)(0.5,0){3}{\line(0,1){3}}
\put(1,1.5){\line(1,0){4}}
\put(4.2,-0.5){\makebox(0,0)}
\put(5.2,-0.5){\makebox(0,0)}
\put(5.5,1.5){\makebox(0,0)}
\put(5.2,3.5){\makebox(0,0)}
\put(4.2,3.5){\makebox(0,0)}
\put(3.5,3.5){\makebox(0,0)}
\put(3,3.5){\makebox(0,0)}
\put(2.5,3.5){\makebox(0,0)}
\end{picture}
\end{minipage}
}
\caption{\textsf{The $\operatorname{B}_3$-arrangement}}
\end{figure}
By \cite{t3}
\[
D(\operatorname{B}_3)\simeq S(-1) \oplus S(-3)\oplus S(-5).
\]
This configuration has $13$ singular points, so if the singularities were in
general position there would only be a two-dimensional space of quartics
passing through the points. However, there are three quadruple
points, and each set of lines through one of these points is a quartic 
vanishing on the singularities. A computation shows that a generic quartic
$C$ in this three dimensional space is smooth, and that $\operatorname{B}_3 \cup C$ is 
quasihomogeneous. 

By Proposition~\ref{HS}, the Hilbert polynomial and Hilbert function 
of the module $\Gamma_{*}\mathcal{O}_{C}(-K_C -H -R)$
agree for $ t < \frac{21}{4}$ and $t > \frac{29}{4}$, and so applying 
Proposition~\ref{HP} and Theorem~\ref{main} we have
\vskip .1in
\begin{center}
\begin{supertabular}{|c|c|c|c|c|c|c|c|c|}
\hline $t$ & $4$ & $5$ & $6$& $7$ & $8$ & $9$ & $10$& $11$ \\
\hline $h^0((t-1)H-K_C-R)$ & $0$ & $0$  & ? & ? & $9$  & $13$ & $17$  & $21$\\
\hline
\end{supertabular}
\end{center}
\vskip .1in
It remains to determine $H^0((t-2)H-R)$ for $t \in \{6,7\}$. The
space $H^0(4H-R)$ consists of quartics through the 
thirteen singular points of $B_3$. As observed above, this space 
has dimension $3$, but $C$ itself is one of the quartics, 
so $h^0(4H-R) =2$.
A direct calculation shows that $h^0(5H-R) = 5$, so 
\[
\HS(\Gamma_{*}\mathcal{O}_{C}(-K_C -H -R),t) = \frac{2t^6+t^7+t^8}{(1-t)^2} = \frac{2t^6-t^7-t^9}{(1-t)^3}.
\]
Since $\Der(\mbox{log}\operatorname{B}_3)(-4) \simeq S(-7)\oplus S(-9)$, taking
global sections in Theorem~\ref{main} and applying 
Proposition~\ref{cmain} shows $\Der(\mbox{log }\operatorname{B}_3 \cup C)$ 
is free, with
\[
\begin{array}{ccc}
\HS(\Der(\mbox{log}\operatorname{B}_3\cup C),t) &= &\HS(\Der(\mbox{log}\operatorname{B}_3)(-4),t)+ \frac{2t^6-t^7-t^9}{(1-t)^3} \\
 & = & \frac{t^7+t^9}{(1-t)^3} + \frac{2t^6-t^7-t^9}{(1-t)^3}\\
 & = &\frac{2t^6}{(1-t)^3}.
\end{array}
\]
\end{exm}
\noindent Example 2.2 of \cite{cmh2} shows in general that $\sing({\mathcal A}) \ne \sing({\mathcal A} \cup C)$.

\noindent{\bf Concluding remarks} Our work raises a number of questions:
\begin{enumerate}
\item Does this generalize to other surfaces? For the Hilbert polynomial
arguments to work, the surface should possess an ample line bundle. 
More generally, does this generalize to higher dimensions? Note that
\cite{cmh2} shows the quasihomogeneous property will be necessary.
\item The Hilbert series of 
$\Gamma_{*}\mathcal{O}_{C}(-K_C-H -R)$ depends solely on a set 
of reduced points on a plane curve. 
If $\mathcal A = \bigcup_{i=1}^r Y_i$ with $Y_i$ 
reduced and irreducible, can an iterated construction using linkage
yield the Hilbert series?
\item  In \cite{liao}, Liao gives a formula relating Chern classes 
of logarithmic vector fields to the Chern-Schwartz-MacPherson class
of the complement, showing that on a surface the two are equal exactly
when the singularities are quasihomogeneous, and in \cite{a}, Aluffi gives
an explicit relation between the characteristic polynomial of an arrangement and the Segre class of the Jacobian scheme. Can one prove Theorem 1.5 using 
these methods?
\end{enumerate}

\noindent {\bf Acknowledgments}:  Macaulay2 computations were
essential to our work. Our collaboration began at the Mathematical
Society of Japan summer school on arrangements; Terao and Yoshinaga
were organizers and Schenck a participant, and we thank the Mathematical
Society of Japan for their generous support.
\renewcommand{\baselinestretch}{1.0}
\small\normalsize 
\pagebreak
\bibliographystyle{amsalpha}

\end{document}